\documentclass[11pt, a4paper]{article}
\usepackage{RAP_Style}
\usepackage{RAP_Defs}

\title{A homotopy classification of $\mathrm{Spin}(7)$-structures with applications to exceptional Riemannian holonomy}
\author{Ra\'{u}l Alvarez-Pati\~{n}o\footnote{Partially supported by CONAHCYT M\'{e}xico under doctoral scholarship no. 701866.}}
\address{Instituto de Matem\'{a}ticas (Cuernavaca) \\ Universidad Nacional Aut\'{o}noma de M\'{e}xico (UNAM)}
\email{ralvarepatino@gmail.com}
\date{}
\subjclass{Primary 53C10, 57R15; Secondary 53C25, 53C27}
\begin{document}
\maketitle
\begin{abstract}    
We use classical obstruction theory \`{a} la Eilenberg-Steenrod to obtain a homotopy classification of $\mathrm{Spin}(7)$-structures on compact $8$-manifolds with abelian fundamental group. As an application we show that a compact, connected Riemannian $8$-manifold with holonomy contained inside the group $\mathrm{Spin}(7)$ has exactly two $\mathrm{Spin}(7)$-structures extending the induced $G_{2}$-structure on the boundary.
\end{abstract}
\keywords{Algo \and algo mas}
\section{Introduction}
Bonan~\cite{bon.SVRG2S7.66}, Gray and Green~\cite{gragree.STTA.70}, and Marisa Fern\'{a}ndez~\cite{fer.CRMS7.86} were among the first authors to consider manifolds of dimension eight endowed with $\lie{Spin}{7}$-structures. Since then, the study of geometric and topological properties of manifolds with $\lie{Spin}{7}$-structures of different types has become a central topic in Riemannian geometry. This paper aims to complete the circle of ideas initiated in~\cite{gragree.STTA.70} by establishing two results. \cref{Thm.mainA} below gives a somewhat general classification of $\lie{Spin}{7}$-structures from the point of view of homotopy theory. Our main result, \cref{Thm.mainB} at the end of this section, establishes the existence of exactly two $\lie{Spin}{7}$-structures on a compact and connected $\lie{Spin}{7}$-manifold.
\begin{thm}
\label{Thm.mainA}
Let $W$ be a compact $8$-manifold with perfect fundamental group. Suppose there is a $\lie{Spin}{7}$-structure on $W$. Then the set of $\lie{Spin}{7}$-structures on $W$ that extend the induced $G_{2}$-structure on the boundary $\partial W$ has a free and transitive action of the group $\cohomrelg{8}{W}{\partial W}{\zn{}_{2}}$.
\end{thm}
Crowley and Nordstr\"{o}m showed that any compact $7$-manifold $M$ with a $G_{2}$-structure admits a $\lie{Spin}{7}$-coboundary, that is, $M$ bounds a spin $8$-manifold endowed with a $\lie{Spin}{7}$-structure~\cite[Lemma 3.4]{crownord.NIG2S.15}. \cref{Thm.mainA} complements this result by giving an explicit description of the set of $\lie{Spin}{7}$-structures extending a $G_{2}$-structure specified on the boundary. On the other hand, the hypothesis for the fundamental group in \cref{Thm.mainA} is motivated by an early result of Joyce stating that a Riemannian $8$-manifold with holonomy inside the group $\lie{Spin}{7}$ is necessarily simply-connected~\cite[Theorem C]{joyce.CMHS7.96}. We provide a deeper reason for our assumption on $\pi_{1}W$ in \cref{Rmk.secdiff}. 
\subsection{Preliminaries and the statement of the main Theorem}
\label{Sub.prelim}
The combined works of Hitchin~\cite{hitch.HARSPI.74}, Wang~\cite{wang.PSPF.89, wang.SCMPS.95}, McInnes~\cite{mcinn.EPSNSC.98}, and Moroianu and Semmelmann~\cite{mose.PSHG.00}, imply that a compact Riemannian manifold with special holonomy is characterized by the existence of non-trivial parallel spinors. Moreover, a compact Riemannian manifold carrying a non-trivial parallel spinor must be Ricci-flat (see Hitchin~\cite[Theorem 1.2]{hitch.HARSPI.74} or Friedrich~\cite[Folgerung 1]{fried.EPSRM.81}) and non-symmetric, because symmetric manifolds have Einstein metrics with non-zero scalar curvature by the results of Besse~\cite[Chapter 7 Section F]{besse.EINMAN}.

Let $H \subset \lie{SO}{n}$ be the holonomy group of an irreducible, non-symmetric, Ricci-flat manifold of dimension $n \geq 3$. Berger's classification~\cite{berg.SLGH.55} implies that $H$ is simply-connected. In particular, $H$ admits a lifting into the universal covering of $\lie{SO}{n}$, the spin group $\lie{Spin}{n}$, with image contained inside the set of all elements that fix a unit spinor in the spin representation of $\lie{Spin}{n}$ (see \cref{Rmk.bigdim} below). On the other hand, a metric with special holonomy reduces the structure group of the tangent bundle of the underlying manifold from $\lie{SO}{n}$ down to $H$. Therefore, the existence of a spin structure supporting a non-zero spinor field is a necessary condition for a compact manifold to admit Ricci-flat metrics of special holonomy.

There are two exceptional cases in the classification of holonomy representations corresponding to Ricci-flat manifolds. The group $G_{2} \subset \lie{SO}{7}$ with a representation in $\rn{7}$ as the automorphism group of the imaginary octonions, and $\lie{Spin}{7}$ acting on $\rn{8}$ via the (real) spin representation $\Delta_{7}$. The copy of $G_{2}$ inside $\lie{Spin}{7}$ coincides with the stabilizer of a unit spinor in the spin representation $\Delta_{7}$. Meanwhile, $\lie{Spin}{7}$ has two liftings into $\lie{Spin}{8}$ corresponding to isotropy subgroups of unit spinors of positive or negative chirality, respectively. A third $\lie{Spin}{7}$-subgroup of $\lie{Spin}{8}$ is obtained as the stabilizer of a unit element in the vector representation $\mathrm{Ad}_{8}$. Varadarajan~\cite[Section 1 Theorem 5]{vara.SP78.01} showed that these subgroups exhaust all possibilities, as they represent every conjugacy class of $\lie{Spin}{7}$-subgroups of $\lie{Spin}{8}$. Moreover, the intersection of any two such $\lie{Spin}{7}$-subgroups gives the copy of $G_{2}$ inside $\lie{Spin}{7}$ associated with the holonomy representation~\cite[Section 3 Theorem 5]{vara.SP78.01}. It is worth mentioning that the two chiral representations $\Delta_{8}^{+}$, $\Delta_{8}^{-}$ and the vector representation $\mathrm{Ad}_{8}$ are permuted by the elements of the outer automorphism group of $\lie{Spin}{8}$. This phenomenon, commonly known as principle of triality, was discovered by \`{E}lie Cartan and can be read from the $S_{3}$-symmetry of the Dynkin diagram of $\lie{SO}{8}$ (see Michelson-Lawson~\cite[Chapter 8]{law.SG}, De Sapio~\cite{dsap.OS8TTA.01} or Baez~\cite[Section 2.4]{baez.TOCT.02}).

The positive spinor bundle of a spin $8$-manifold $W$ is an orientable $8$-plane bundle associated with a spin structure on $W$ via the right-handed spin representation $\Delta_{8}^{+}$. We say that $W$ has a geometric $\lie{Spin}{7}$-structure if there is a positive spinor bundle $S^{+}_{W}$ over $W$ carrying a right-handed spinor field $\psi \in \Gamma(S^{+}_{W})$ with norm $\abs{\psi} = 1$ everywhere on $W$. A $\lie{Spin}{7}$-structure is the homotopy class represented by a path $\psi_{t} \in \Gamma(S^{+}_{W})$ with $\abs{\psi_{t}} = 1$ for all $t \in [0,1]$ and at every point of $W$ (cf \cref{Def.red}). Then $W$ has a $\lie{Spin}{7}$-structure if and only if the Euler class $e(S^{+}_{W}) = 0$. Gray and Green~\cite[Theorem 3.4]{gragree.STTA.70} found that $e(S^{+}_{W})$ can be expressed in terms of Pontrjagin classes $p_{1}(TW), p_{2}(TW)$ and the Euler class $e(TW)$ of the tangent bundle of $W$ according to
\begin{equation*}
\label{Eq.graygreen}
16 \, e(S^{+}_{W}) = 4 p_{2}(TW) - p_{1}^{2}(TW) + 8 e(TW).
\end{equation*}
This means that the existence of a $\lie{Spin}{7}$-structure on $W$ is a purely topological condition on the tangential information of $W$ that does not depend on a specific choice of spin structure. Accordingly, a geometric $\lie{Spin}{7}$-structure on $W$ is defined by a pair $(\Omega, g)$, where $\Omega$ is a $4$-form on $W$ depending quadratically on the unit spinor field $\psi$ (see Lucia Mart\'{i}n-Merch\'{a}n~\cite[Proposition 2.4]{merch.SCS7.20}), and $g$ is a Riemannian metric determined by $\Omega$ in a highly non-linear fashion (see Karigiannis~\cite[Theorem 4.3.3]{kar.DG2S7S.05}). Note that $\Omega^{2}$ is non-vanishing and consequently determines an orientation for $W$. A well-known result of Marisa Fern\'{a}ndez~\cite[Theorem 5.3]{fer.CRMS7.86} implies that the metric $g$ has $\lie{Spin}{7}$-holonomy if and only if $\Omega$ is closed, and this happens precisely when $\psi$ is parallel. In such case, one says that the corresponding $\lie{Spin}{7}$-structure is torsion-free.

The first non-compact examples of Riemannian manifolds with $\lie{Spin}{7}$-holonomy were discovered by Bryant~\cite{brya.MEH.87}. Later on, Bryant and Salamon~\cite{brsa.OCCMEH.89} extended the catalogue to include complete but still non-compact manifolds. The compact case was finally settled by Joyce in~\cite{joyce.CMHS7.96, joyce.NCMS8.99}. Those manifolds constructed by Joyce were the first examples of Ricci-flat $8$-manifolds of non-K\"{a}hler type. Another key result implies that the holonomy of a compact, simply-connected, $8$-manifold $W$ with a torsion-free $\lie{Spin}{7}$-structure is controlled by the index of the Dirac operator. More precisely, the $\hat{A}$-genus $\hat{A}(W)$ determines the holonomy $H$ by the following criteria; $H = \lie{Spin}{8-k}$ if and only if $\hat{A}(W) = k$, for $k = 1, 2, 3, 4$~\cite[Theorem C]{joyce.CMHS7.96}. This together with the Cheeger-Gromoll splitting Theorem~\cite{chegrom.STMNNR.71} implies that $W$ must be simply-connected. This result by Joyce is the strongest obstruction to the existence of Riemannian metrics with $\lie{Spin}{7}$-holonomy known to this day. Combining \cref{Thm.mainA} with~\cite[Theorem C]{joyce.CMHS7.96} we get our main result.
\begin{thm}
\label{Thm.mainB}
Let $W$ be a compact, connected $8$-manifold with boundary $\partial W$. Suppose that $W$ has a Riemannian metric with holonomy contained inside the group $\lie{Spin}{7}$. Then there are exactly two $\lie{Spin}{7}$-structures on $W$ that extend the induced $G_{2}$-structure on the boundary.
\end{thm}
\begin{org*}
In \cref{Sec.red} we formulate the existence and uniqueness problem of reductions of the structure group of a vector bundle as a lifting problem associated with a canonical fibration between the corresponding classifying spaces. \cref{Sec.obs} provides a concise summary of the main results of obstruction theory. We emphasize \cref{Sub.diffcocy}, where the comparison between two reductions of the structure group is implemented by an affine difference function. In \cref{Sec.spinrep} we recall all the necessary notions of spin geometry in dimensions $7$ and $8$ for later use in \cref{Sec.sphfibstab} where explicit identifications between isotropy subgroups of unit spinors and certain spin subgroups are given. The proof of \cref{Thm.mainA} and some final remarks for future work are presented in \cref{Sec.spin7}. 
\end{org*}
\begin{akn*}
The author is very grateful to Professors; Gregor Weingart, Rustam Sadykov, H\'{e}ctor Hugo Garc\'{i}a Compe\'{a}n and Pablo Su\'{a}rez Serrato for their continuous advice, help, and encouragement throughout author's career. Finally, the author would like to take the opportunity to thank Natalia Contreras for the unconditional support that made this work possible.
\end{akn*}
\section{Reductions of the structure group of a vector bundle}
\label{Sec.red}
Let us consider a pair of compact Lie groups $H$ and $G$ related by a smooth homomorphism $\iota : H \rightarrow G$ with kernel $\kr \iota$. The induced map $\bg{\iota} : \bg{H} \rightarrow \bg{G}$ between classifying spaces defines a Hurewicz fibration 
\begin{equation}
\label{Eq.hurfib}
F \longrightarrow \bg{H} \stackrel{\bg{\iota}}{\longrightarrow} \bg{G}.
\end{equation}
The homotopy fibre $F = \mathrm{Hofib}(\bg{\iota})$ is homotopy equivalent to the quotient space $G/H$ if $\kr \iota$ is trivial, and is given by the classifying space $\bg{\kr \iota}$ when $\iota$ is surjective. Milnor's construction ~\cite[Section 3]{mil.CUB2.56} together with our compactness assumption ensures that $\bg{H}$ and $\bg{G}$ are countable CW complexes. The same is true for the corresponding universal spaces $\mathrm{E}H$ and $\mathrm{E}G$.  From now on, $X$ will denote a finite CW complex.
\begin{definition}
\label{Def.red}
Let $S : X \rightarrow \bg{G}$ be the classifying map of a vector bundle $S_{X}$ over $X$ with structure group $G$. A geometric $H$-reduction for $S_{X}$ is a homotopy lifting $\hat{S} : X \rightarrow \bg{H}$ of the classifying map $S$ along the projection $\bg{\iota}$. An $H$-structure for $S_{X}$ is the homotopy class represented by a continuous path of geometric $H$-reductions for $S_{X}$. We write $SX$ to denote the total space of the bundle $S_{X}$.
\end{definition}
\begin{remark}
A homotopy lifting of $S$ along a projection $\bg{\iota}$ is a map $\hat{S}$ such that the composition $\bg{\iota} \circ \hat{S}$ is homotopic to $S$. From now on we will use the term lifting in place of homotopy lifting.
\end{remark}
\begin{remark}
\label{Rmk.verthom}
Denote by $S^{\ast} \bg{H} \rightarrow X$ the pullback of \cref{Eq.hurfib} along $S : X \rightarrow \bg{G}$. Note that a lifting for $S$ and a section of $S^{\ast} \bg{H} \rightarrow X$ define each other. For this reason, we identify the set of all geometric $H$-reductions of $S_{X}$ with the set of sections $\Gamma(S^{\ast} \bg{H})$ of the fibration $S^{\ast} \bg{H} \rightarrow X$. Moreover, two liftings for $S$ are homotopic if and only if the corresponding sections of $S^{\ast} \bg{H}$ are homotopic. In general, two sections of a given fibration are homotopic if and only if they are vertically homotopic~\cite[Lemma 1.1]{jamtho.NCCS.66}; that is, they are homotopic through a path of sections. Therefore, if we denote by $H(S_{X})$ the set of all $H$-reductions of $S_{X}$, then there is a natural bijection
\begin{equation}
\label{Eq.hlift}
H(S_{X}) \simeq \pi_{0} \Gamma(S^{\ast} \bg{H}),
\end{equation}
where $\pi_{0}\Gamma(S^{\ast} \bg{H})$ denotes the set of connected components of the space of sections $\Gamma(S^{\ast} \bg{H})$. From now on we make no distinction between homotopies and vertical homotopies when dealing with liftings.
\end{remark}
The following example serves as the basic model for geometric $\lie{Spin}{7}$-structures and $G_{2}$-structures, respectively.
\begin{example}
\label{Ex.sphfib}
Suppose that $G$ is a compact and connected Lie group, and let $\rho : G \rightarrow \lie{SO}{m+1}$ be a smooth and faithful representation of $G$ into the Euclidean space $\rn{m+1}$. Define $H \subset G$ as the isotropy subgroup of the first coordinate vector $e_{0} \in \rn{m+1}$. Consider the associated $G$-vector bundle $S_{X}$ corresponding to $\rho$. Then $S_{X}$ has an inner product $\braket{\cdot}{\cdot}$ and an orientation induced by $\rho$. A geometric $H$-reduction for $S_{X}$ is completely determined by a choice of nowhere vanishing section $s : X \rightarrow SX$ normalized so that $\braket{s}{s} = \abs{s}^{2} = 1$ everywhere. Such a section exists precisely when the Euler class $e(S_{X})$ equals zero. Moreover, any two such sections induce the same $H$-structure for $S_{X}$ if and only if they are homotopic through a path of sections of the associated $m$-sphere bundle formed by all unit vectors of $S_{X}$. If the group $G$ acts transitively on the sphere $S^{m} \subset \rn{m+1}$, then there is a diffeomorphism identifying the quotient space $G/H$ with $S^{m}$. In this case, the associated fibration \cref{Eq.hurfib} is spherical and takes the form
\begin{equation}
\label{Eq.sphfib}
S^{m} \rightarrow \bg{H} \rightarrow \bg{G}.
\end{equation}
Montgomery and Samelson~\cite{mon.TGS.43} showed that the classical groups $\lie{SO}{n}$, $\lie{U}{n}$, $\lie{SU}{n}$, $\lie{Sp}{n}$, $\lie{Sp}{n} \cdot \lie{SO}{2}$, $\lie{Sp}{n} \cdot \lie{Sp}{1}$, together with the exceptional groups; $G_{2}, \lie{Spin}{7}$, and $\lie{Spin}{9}$ are the only compact Lie groups acting effectively and transitively on a finite-dimensional sphere.
\end{example}
\section{Elements of obstruction theory}
\label{Sec.obs}
In this section we present without proof the main results of obstruction theory. The reader is refereed to ~\cite[Part III]{steen.TFB} for details. In \cref{Def.primdiff} the primary difference $D_{0}$ between two $H$-structures is introduced. Then \cref{Lem.transact} characterizes a free and transitive action in terms of the properties satisfied by $D_{0}$. Perhaps the novel feature in our exposition is the application of \cref{Lem.transact} in the classification of reductions of the structure group of a vector bundle (see \cref{Cor.ftransact}).
\subsection{The obstruction cocycle and the existence of $H$-structures}
\label{Sub.obscocy}
For us, a relative CW complex $(X, Y)$ will mean a CW pair $(X, Y)$ with $X$ finite, and where $Y \subset X$ is a proper, closed, and possibly empty subspace. Consider a $G$-vector bundle $S_{X}$ over $X$ classified by a map $S : X \rightarrow \bg{G}$ as in \cref{Def.red}. Motivated by \cref{Ex.sphfib}, we further assume that the homotopy fibre $F = \mathrm{Hofib}(\bg{\iota})$ in equation \cref{Eq.hurfib} is homotopy equivalent to a $7$-sphere $S^{7}$. This hypothesis also allows us to use cohomology with constant coefficients.

Given an integer $k \geq 0$, write $X^{(k)}$ to denote the $k$-skeleton of $X$. Suppose we have defined a partial lifting $\hat{S} : Y \cup X^{(k)} \rightarrow \bg{H}$ for $S$ over the subcomplex $Y \cup X^{(k)}$. Following~\cite[Theorem 4.11 p 136]{cohen.TFB}, every $(k+1)$-cell $e^{k+1}$ attached to $X$ via $\alpha : \partial e^{k + 1} \rightarrow X^{(k)}$ induces a commutative diagram
\begin{equation}
\label{Eq.obstprob}
\begin{split}
\xymatrix@R=8mm@C=8mm{
\partial e^{k+1} \ar[r]^{\alpha} \ar@{^{(}->}[d]  & Y \cup X^{(k)} \ar[r]^{\hat{S}} \ar@{^{(}->}[d] & \bg{H} \ar[d]^{\bg{\iota}} \\
e^{k+1} \ar@{^{(}->}[r] & Y \cup X^{(k+1)} \ar[r]_{S} & \bg{G} 
}
\end{split}
\end{equation}
where the leftmost square of \cref{Eq.obstprob} is a push-out diagram, and the hooked arrows represent inclusion maps. Note that the composition $\bg{\iota} \circ \hat{S} \circ \alpha$ factors though $e^{k+1}$. By identifying $e^{k+1}$ with the cone over the $k$-sphere $\partial e^{k+1}$ we immediately get a null-homotopy for $\bg{\iota} \circ \hat{S} \circ \alpha$. This in turn implies that $\hat{S} \circ \alpha$ is homotopic to a map with values in the fibre $S^{7}$, since $\bg{\iota}$ is the projection of a fibration. Thus, $\hat{S} \circ \alpha$ represents a well defined class $[\hat{S} \circ \alpha]$ in the $k^{th}$-homotopy group $\pi_{k}S^{7}$. By construction, $\hat{S} : Y \cup X^{(k)} \rightarrow \bg{H}$ extends over $Y \cup X^{(k)} \cup_{\alpha} e^{k+1}$ if and only if $[\hat{S} \circ \alpha] = 0$.
\begin{definition}
\label{Def.obscocycle}
The $\pi_{k}S^{7}$-valued cochain $\mathcal{O}_{k}(\hat{S}) \in C^{k+1}(X,Y;\pi_{k}S^{7})$ defined by the correspondence
\begin{equation}
\label{Eq.obscoch}
\mathcal{O}_{k}(\hat{S})(e^{k+1}) = [\hat{S} \circ \alpha]
\end{equation}
is the complete obstruction to extending an $H$-structure on $S_{X}$ from $Y \cup X^{(k)}$ to $Y \cup X^{(k+1)}$. Moreover, \cref{Eq.obscoch} is a cocycle by~\cite[Theorem 32.4 p 167]{steen.TFB}. For this reason $\mathcal{O}_{k}(\hat{S})$ is called the obstruction cocycle.
\end{definition}
\begin{remark}
\label{Rmk.highobs}
By definition, there are no obstructions for extending $\hat{S}$ from $Y$ to $Y \cup X^{(k)}$ when $k < 7$. The cohomology class $c_{0}(S_{X}) \in \cohomrelg{8}{X}{Y}{\pi_{7}S^{7}}$ represented by the corresponding obstruction cocycle \cref{Eq.obscoch} is the first potential obstruction to our extension problem. In general, there might be higher obstruction classes $c_{k}(S_{X}) \in \cohomrelg{8 +k}{X}{Y}{\pi_{7+k}S^{7}}$ associated with other non-zero values of $\pi_{\ast}S^{7}$. It follows that the classifying map $S : X \rightarrow \bg{G}$ admits a lifting to $\bg{H}$ along $\bg{\iota}$ if and only if the cohomology classes $c_{k}(S_{X}) = 0$ for $k=0, 1, \dots \dim X$. In particular, $c_{0}(S_{X})$ is the only obstruction to the existence of an $H$-reduction for $S_{X}$ when $\dim X \leq 7$. The obstruction classes are all independent of the specific choice of partial lifting. A change in $\hat{S} : Y \cup X^{(k)} \rightarrow \bg{H}$ over lower-dimensional skeleta alters the obstruction cocycle \cref{Eq.obscoch} by a coboundary~\cite[Lemma 32.1 p 166]{steen.TFB}.
\end{remark}
\subsection{The difference cocycle between two $H$-structures}
\label{Sub.diffcocy}
Now we compare two liftings $\hat{S}_{0} , \hat{S}_{1} : X \rightarrow \bg{H}$ of the classifying map $S : X \rightarrow \bg{G}$ that coincide on the subcomplex $Y \subset X$. More precisely, we will determine all the obstructions for the existence of a homotopy between $\hat{S}_{0}$ and $\hat{S}_{1}$ relative to $Y$. This is achieved essentially by applying the construction of the obstruction cocycle introduced in \cref{Def.obscocycle} to a homotopy on a convenient subspace of the product complex $X\times I$, where $I =[0,1]$ is the unit interval.

Recall that the $k$-skeleton of $X \times I$ is given by $(X \times I)^{(k)} = (X^{(k-1)} \times I) \cup (X^{(k)} \times \partial I)$ for all integers $k \geq 1$. We define the extended CW pair $(\tilde{X}, \tilde{Y})$ by
\begin{equation*}
\label{Eq.int}
\tilde{X} = X \times I, \quad \tilde{Y} = \left( Y \times I \right) \cup \left( X \times \partial I \right).
\end{equation*}
The crucial observation is that a map $\tilde{S} : \tilde{Y} \cup \tilde{X}^{(k)} \rightarrow \bg{H}$ is equivalent to a pair of continuous functions $\hat{S}_{0}, \hat{S}_{1} : X \rightarrow \bg{H}$ together with a homotopy $\hat{H}_{t} : Y \cup  X^{(k-1)}\rightarrow \bg{H}$, defined for $t \in I$, between the restrictions $\hat{S}_{0} \, | \, Y \cup X^{(k-1)}$ and $\hat{S}_{1} \, | \, Y \cup X^{(k-1)}$, and which fixes the subspace $Y$.

As in the previous section, we continue to assume that the homotopy fibre of $\bg{\iota}$ is equivalent to $S^{7}$ (see equation \cref{Eq.sphfib}). This hypothesis certainly guarantees that $\hat{S}_{0} \, | \, Y \cup X^{(k-1)}$ and $\hat{S}_{1} \, | \, Y \cup X^{(k-1)}$ are homotopic for all $k \leq 6$. Moreover, there is a homotopy $\hat{H}_{t}$ connecting $\hat{S}_{0} \, | \, Y \cup X^{(6)}$ and $\hat{S}_{1} \, | \, Y \cup X^{(6)}$ relative to $Y$ given by~\cite[Lemma 36.2 p 181]{steen.TFB}. Thus, we have all the ingredients to define a map $\tilde{S} : \tilde{Y} \cup \tilde{X}^{(7)} \rightarrow \bg{H}$ and the corresponding obstruction cocycle $\mathcal{O}_{7}(\tilde{S}) \in C^{8}(\tilde{X}, \tilde{Y} ; \pi_{7}S^{7})$ as in \cref{Def.obscocycle}. To obtain an obstruction on the original pair $(X,Y)$ we use the isomorphism
\begin{equation*}
\label{Eq.crosscoh}
\times \bar{\mathrm{I}} : C^{7}(X, Y ; \pi_{7}S^{7}) \rightarrow C^{8}(\tilde{X}, \tilde{Y} ; \pi_{7}S^{7})
\end{equation*}
induced by the cross product with the generator $\bar{\mathrm{I}} \in C^{1}(I; \zn{})$ of the group of $1$-cochains on the unit interval. We define the difference cochain $\dbraket{\hat{S}_{1}}{\hat{H}_{t}}{\hat{S}_{0}} \in C^{7}(X, Y ; \pi_{7}S^{7})$ implicitly by
\begin{equation}
\label{Eq.difcoch}
\dbraket{\hat{S}_{1}}{\hat{H}_{t}}{\hat{S}_{0}} \times \bar{\mathrm{I}} = \mathcal{O}_{7}(\tilde{S}) - \mathcal{O}_{7}(\hat{S}_{0}) \times \bar{\mathrm{0}} -  \mathcal{O}_{7}(\hat{S}_{1}) \times \bar{\mathrm{1}}
\end{equation}
where $\bar{\mathrm{0}}, \bar{\mathrm{1}} \in C^{0}(I; \zn{})$ are generators chosen so that $\delta \bar{\mathrm{0}} = - \bar{\mathrm{I}}$, and $\delta \bar{\mathrm{1}} = \bar{\mathrm{I}}$. Note that we have subtracted the term $\mathcal{O}_{7}(\hat{S}_{0}) \times \bar{\mathrm{0}} + \mathcal{O}_{7}(\hat{S}_{1}) \times \bar{\mathrm{1}}$ to make the right hand side of equation \cref{Eq.difcoch} a relative cochain. Applying the coboundary operator to the expression \cref{Eq.difcoch} we obtain
\begin{equation*}
\delta \dbraket{\hat{S}_{1}}{\hat{H}_{t}}{\hat{S}_{0}} = \mathcal{O}_{7}(\hat{S}_{0}) - \mathcal{O}_{7}(\hat{S}_{1}).
\end{equation*}
Given that both $\hat{S}_{0}$ and $\hat{S}_{1}$ are global liftings of the classifying map $S : X \rightarrow \bg{G}$, we have $\mathcal{O}_{7}(\hat{S}_{0}) =0 = \mathcal{O}_{7}(\hat{S}_{1})$. This shows that $\dbraket{\hat{S}_{1}}{\hat{H}_{t}}{\hat{S}_{0}}$ is a cocycle of degree $7$. Moreover, if $\hat{K}_{t}$ is another homotopy between $\hat{S}_{0} | X^{(6)}$ and $\hat{S}_{1}| X^{(6)}$ relative to $Y$, then $\dbraket{\hat{S}_{1}}{\hat{H}_{t}}{\hat{S}_{0}}$ and $\dbraket{\hat{S}_{1}}{\hat{K}_{t}}{\hat{S}_{0}}$ differ by a coboundary. Hence, the cohomology class represented by $\dbraket{\hat{S}_{1}}{\hat{H}_{t}}{\hat{S}_{0}}$ is independent of $\hat{H}_{t}$.
\begin{definition}
\label{Def.primdiff}
Consider the $G$-vector bundle $S_{X}$ over a finite CW complex $X$ classified by a map $S : X \rightarrow \bg{G}$ as in \cref{Def.red}. Let $\hat{S}_{0} , \hat{S}_{1} : X \rightarrow \bg{H}$ be two lifting for $S$ along the induced map $\bg{\iota}$ with homotopy fibre equivalent to $S^{7}$. The primary difference $D_{0}(\hat{S}_{1}, \hat{S}_{0})$ between $\hat{S}_{0}$ and $\hat{S}_{1}$ is the cohomology clas
\begin{equation}
\label{Eq.primdiff}
D_{0}(\hat{S}_{1}, \hat{S}_{0}) = \left[ \dbraket{\hat{S}_{1}}{\hat{H}_{t}}{\hat{S}_{0}} \right]  \in \cohomrelg{7}{X}{Y}{\pi_{7}S^{7}}.
\end{equation}
\end{definition}
\begin{remark}
\label{Rmk.secdiff}
Two $H$-structures represented by maps $\hat{S}_{0}, \hat{S}_{1} : X \rightarrow \bg{H}$ that agree on the sub-complex $Y$ are homotopic up to the $6$-skeleton, and a given homotopy extends over $Y \cup X^{(7)}$ if and only if $D_{0}(\hat{S}_{1}, \hat{S}_{0}) = 0$. In such a case, the next obstruction for an homotopy between $\hat{S}_{0}$ and $\hat{S}_{1}$ occurs on the $8$-skeleton, and it is measured by a secondary difference $D_{1}(\hat{S}_{1}, \hat{S}_{0}) \in \cohomrelg{8}{X}{Y}{\pi_{8}S^{7}}$ defined by the analogue of formula \cref{Eq.difcoch} in dimension eight. To pass to the next stage of this iterative process, it is necessary and sufficient that $D_{1}(\hat{S}_{1}, \hat{S}_{0}) = 0$. This procedure continues, one skeleton at a time, until the higher-dimensional skeleton is reached, provided that the corresponding difference is zero.
\end{remark}
\begin{remark}
\label{Rmk.properties}
The primary difference $D_{0}$, as any other subsequent difference, solely depends on the homotopy class of its arguments. Let $\Gamma = H(S_{X})$ be the set of all $H$-reductions for $S_{X}$ (see equation~\cref{Eq.hlift}) and set $\mathcal{H} = \cohomrelg{7+k}{X}{Y}{\pi_{7+k}S^{7}}$. Then each partial difference determines a well-defined function $D : \Gamma \times \Gamma \rightarrow \mathcal{H}$ with the following properties. First, it satisfies the cocycle condition; $D(\hat{S}_{2}, \hat{S}_{0}) = D(\hat{S}_{2}, \hat{S}_{1}) + D(\hat{S}_{1}, \hat{S}_{0})$ by~\cite[Theorem 36.6 p 182]{steen.TFB}. Second, $D(\hat{S}_{1}, \hat{S}_{0}) = 0$ if and only if $\hat{S}_{0}, \hat{S}_{1} \in \Gamma$ are homotopic~\cite[Theorem 36.4 p 181]{steen.TFB}. Finally, given $\hat{S}_{0} \in \Gamma$ and $h \in \mathcal{H}$, there is a unique $\hat{S}_{1} \in \Gamma$ such that $D(\hat{S}_{1}, \hat{S}_{0}) = h$~\cite[Lemma 33.9 p 174]{steen.TFB}.
\end{remark}
\begin{lemma}
\label{Lem.transact}
Let $\mathcal{H}$ be an abelian group and $\Gamma$ any set. Then a free and transitive action $\mathfrak{a} : \mathcal{H} \times \Gamma \rightarrow \Gamma$ is equivalent to the existence of an affine difference function $D : \Gamma \times \Gamma \rightarrow \mathcal{H}$ satisfying the following properties.
\begin{enumerate}[label={\alph*)}]
\item\label{It.adp1} For every triple $x,y,z \in \Gamma$, the cocycle condition $D(x,z) = D(x,y) + D(y,z)$ holds.
\item\label{It.adp2} $D(x,y) = 0$ if and only if $x=y$.
\item\label{It.adp3} Given $x \in \Gamma$ and $h \in \mathcal{H}$, there is a unique solution $y \in \Gamma$ to the equation $D(x,y) = h$.
\end{enumerate}
\end{lemma}
\begin{proof}
Assume we have an action $\mathfrak{a} : \mathcal{H} \times \Gamma \rightarrow \Gamma$, denoted simply by $\mathfrak{a}(h,x) = hx$. If $\mathfrak{a}$ is free and transitive, then every pair $(x,y) \in \Gamma \times \Gamma$ defines a unique element $D(x,y) \in \mathcal{H}$ such that $y=D(x,y)x$. This gives us the affine difference function $D : \Gamma \times \Gamma \rightarrow \mathcal{H}$. Now we show that $D$ satisfies properties \cref{It.adp1,It.adp2,It.adp3}. First we check the cocycle condition. For every triple $x,y,z \in \Gamma$ we have expressions $y = D(x,y) x$ and $z = D(y,z) y$. Substituting we get
\begin{equation*}
z = D(y,z)y = D(y,z) \left( D(x,y)x \right) = \left( D(x,y) + D(y,z) \right)x.
\end{equation*}
The uniqueness of $D(x,z) \in \mathcal{H}$ satisfying $z = D(x,z)x$ implies that $D(x,z) = D(x,y) + D(y,z)$. Now we verify property \cref{It.adp2}. Since $\mathfrak{a}$ is free, $0 \in \mathcal{H}$ is the only group element fixing a given $x \in \Gamma$; that is, $D(x,y) = 0$ if and only if $x=y$. Finally property \cref{It.adp3}. 
Let $x \in \Gamma$ and $h \in \mathcal{H}$ and suppose that $D(x,z)=h$ has two solutions $z=y$ and $z=y'$. Then $D(x,y) = D(x,y')$. After straightforward manipulations using $D(x,y) = -D(y,x)$ one gets $D(y,y') = 0$. Hence $y = y'$ by property \cref{It.adp2}.

Conversely, if an affine difference $D : \Gamma \times \Gamma \rightarrow \mathcal{H}$ satisfying properties \cref{It.adp1,It.adp2,It.adp3} is given, we define a function $\mathfrak{a} : \mathcal{H} \times \Gamma \rightarrow \Gamma$ by the condition $\mathfrak{a}(h,x) = y$, where $y \in \Gamma$ is the unique solution to $D(x,y) = h$ provided by \cref{It.adp3}. Again we denote the value of $\mathfrak{a}(h,x)$ by $hx$. We claim that $\mathfrak{a}$ is an action. The trivial element $0 \in \mathcal{H}$ satisfies $0x=x$ for all $x \in \Gamma$; if $0x=y$, then $D(x,y) = 0$ and property \cref{It.adp2} implies that $y=x$. Finally let us show that the compatibility condition $k(hx) = (h+k)x$ is satisfied. Set $y = hx$ and $z = ky$; in other words, $D(x,y) = h$ and $D(y,z) = k$. Then $h+k = D(x,y) + D(y,z)$ and by the cocycle condition we have $h+k = D(x,z)$; that is, $(h+k) x = z = k(hx)$.
\end{proof}
\begin{corollary}
\label{Cor.ftransact}
Let $(X,Y)$ be a relative CW complex such that $\dim X \leq 8$ and with $\cohomrelg{7}{X}{Y}{\zn{}} = 0$. Consider a spherical fibration 
\begin{equation*}
S^{7} \longrightarrow \bg{H} \stackrel{\bg{\iota}}{\longrightarrow} \bg{G},
\end{equation*}
and let $S_{X}$ be a vector bundle over $X$ classified by a map $S : X \rightarrow \bg{G}$ as in \cref{Ex.sphfib}. In addition, assume that $S_{X}$ admits an $H$-structure in the sense of \cref{Def.red}. Then the set of $H$-structures on $S_{X}$ extending a given $H$-structure on $Y$ has a free and transitive action of the group
\begin{equation}
\label{Eq.ftransact}
\cohomrelg{8}{X}{Y}{\zn{}_{2}}.
\end{equation}
\end{corollary}
\begin{proof}
Our hypothesis for the cohomology group $\cohomrelg{7}{X}{Y}{\zn{}}$ implies that two $H$-structures for $S_{X}$ can only differ over the $8$-skeleton $X^{(8)}$. In this case, the secondary difference between two such structures is an element of the group \cref{Eq.ftransact}. Therefore, the properties satisfied by the secondary difference listed at the end of \cref{Rmk.properties} together with \cref{Lem.transact} give \cref{Cor.ftransact}.
\end{proof} 
\section{Real spin representations for $\lie{Spin}{7}$ and $\lie{Spin}{8}$}
\label{Sec.spinrep}
The Clifford algebra $\clif{n}$ associated with the Euclidean space $(\rn{n} , g_{0})$ is defined as the associative $\rn{}$-algebra with unit $1 \in \clif{n}$ generated by vectors $v, w \in \rn{n}$ subject to the Clifford relation 
\begin{equation}
\label{Eq.clifrel}
vw + wv = -2 g_{0}(v,w) 1.
\end{equation}
If $\lbrace e_{1}, \dots, e_{n} \rbrace$ is the canonical orthonormal basis of $\rn{n}$ given by the coordinate vectors, then the set of all monomials $\xi = e_{\mu_{1}} \cdots e_{\mu_{k}}$ with $1 \leq \mu_{1} < \cdots < \mu_{k} \leq n$ generates $\clif{n}$ as a real algebra. In particular,
 $\dim_{\rn{}} \clif{n} = 2^{n}$. The antipodal involution $\alpha : \clif{n} \rightarrow \clif{n}$, $\alpha(\xi) = (-1)^{k} \xi$, decomposes $\clif{n}$ into the direct sum of two eigenspaces $\clifp{n}$ and $\clifi{n}$. The even part $\clifp{n}$ is the subalgebra of all products with an even number of factors. The odd part is formed by elements $\xi \in \clif{n}$ such that $\alpha(\xi) = - \xi$. The orthogonal decomposition $\rn{n+1} = \rn{} \langle e_{0} \rangle \oplus \rn{} \langle e_{1} , \dots ,e_{n} \rangle$, where $e_{0} \in \rn{n+1}$ is the first coordinate vector, induces an isomorphism $P_{n} : \clif{n} \rightarrow \clifp{n+1}$
\begin{equation}
\label{Eq.eveniso}
P_{n}(\xi) = e_{0} e_{\mu_{1}} \cdots e_{0}e_{\mu_{k}}.
\end{equation}
The spin group $\lie{Spin}{n}$ is the subgroup of $\clif{n}$ generated by all monomials $\zeta = v_{1} \cdots v_{2k} \in \clifp{n}$, where $g_{0}(v_{\mu}, v_{\mu}) = 1$ for all $\mu = 1, \dots, 2k$. The Clifford relation \cref{Eq.clifrel} implies that $\zeta^{-1} = v_{2k} \cdots v_{1}$, and from this it follows that the linear transformation given by $x \mapsto \zeta x \zeta^{-1}$ for all $x \in \rn{n}$ is completely determined by an expression of the form $v x v^{-1}$, where $v \in S^{n-1}$. Using equation \cref{Eq.clifrel}, one easily shows that $- v x v^{-1}$ is a reflection across the orthogonal complement of $v$. Moreover, all elements in the special orthogonal group $\lie{SO}{n}$ are obtained as compositions of an even number of reflections. Therefore, every $\zeta \in \lie{Spin}{n}$ defines a rotation of $\rn{n}$ by the formula $\zeta x \zeta^{-1}$.
\begin{definition}
\label{Def.vectrep}
The vector representation $\mathrm{Ad}_{n} : \lie{Spin}{n} \rightarrow \lie{SO}{n}$ is defined by $\mathrm{Ad}_{n}(\zeta) x =\zeta x \zeta^{-1}$, where $\zeta \in \lie{Spin}{n}$ and $x \in \rn{n}$. By the classical Cartan-Dieudonn\'{e} theorem, $\mathrm{Ad}_{n}$ is surjective. On the other hand, $\kr \mathrm{Ad}_{n}$ is generated by $-1 \in \lie{Spin}{n}$. This shows that $\mathrm{Ad}_{n}$ is a double covering that corresponds to the universal cover of $\lie{SO}{n}$ when $n \geq 3$.
\end{definition}
\begin{remark}
\label{Rmk.descend}
Every representation of $\lie{SO}{n}$ pulls-back under $\mathrm{Ad}_{n}$ to a representation of $\lie{Spin}{n}$. However, a representation $\lie{Spin}{n} \rightarrow \lie{GL}{V}$ descends to $\lie{SO}{n}$ if and only if it maps the generator $-1 \in \kr \mathrm{Ad}_{n}$ to the identity map $\mathrm{id} \in \lie{GL}{V}$. By linearity, any representation of $\clif{n} \rightarrow \mathrm{End}_{\rn{}}(V)$ sends $-1 \in \clif{n}$ to $- \mathrm{id} \in \mathrm{End}_{\rn{}}(V)$. Therefore, the restriction of any representation for the Clifford algebra $\clif{n}$ to the spin group $\lie{Spin}{n} \subset \clifp{n}$ gives a representation that does not descend to $\lie{SO}{n}$. 
\end{remark}
In eight dimensions $\clif{8}$ is isomorphic to the algebra $\rn{}(16)$ of $16 \times 16$ matrices with real entries~\cite[Theorem 4.3 p 27]{law.SG}. Hence, up to isomorphism, $\clif{8}$ has a unique real irreducible representation 
\begin{equation}
\label{Eq.clif8}
c : \clif{8} \rightarrow \mathrm{End}_{\rn{}}(\rn{16})
\end{equation}
that endows $\rn{16}$ with a left module structure over $\clif{8}$. From now on, we write $S_{8}$ to denote the Clifford module $(\rn{16}, c)$. The linear transformation $c_{\xi} : S_{8} \rightarrow S_{8}$ corresponding to an element $\xi \in \clif{8}$ under \cref{Eq.clif8} is called Clifford multiplication by $\xi$. 

Consider the oriented volume element $\omega_{8} \in \clif{8}$ defined as the ordered product of all the coordinate vectors in $\rn{8}$, that is, $\omega_{8} = e_{1} \cdots e_{8}$. It belongs to $\lie{Spin}{8}$, commutes with the even part $\clifp{8}$ and satisfies $\omega_{8}^{2} = 1$. Therefore, Clifford multiplication by $\omega_{8}$ splits the irreducible module $S_{8}$ as a direct sum of eigenspaces
\begin{equation*}
\label{Eq.quiral}
S_{8} = S_{8}^{+} \oplus S_{8}^{-},
\end{equation*}
where $S_{8}^{+}$ and $S_{8}^{-}$ correspond to the positive and negative eigenvalue, respectively. Note that Clifford multiplication $c_{v} : S_{8} \rightarrow S_{8}$ interchanges $S_{8}^{+}$ and $S_{8}^{-}$ for every non-zero vector $v \in \rn{8}$. In particular, $S_{8}^{+}$ and $S_{8}^{-}$ are isomorphic subspaces of real dimension $8$. Since $\omega_{8}$ commutes with all even elements of $\clif{8}$, $S_{8}^{+}$ and $S_{8}^{-}$ are both invariant under the action of the subalgebra $\clifp{8}$. 
\begin{definition}
\label{Def.rightspin}
The positive spin representation $\Delta_{8}^{+} : \lie{Spin}{8} \rightarrow \lie{GL}{S_{8}^{+}}$ is obtained by restricting the isomorphism $c : \clif{8} \rightarrow \mathrm{End}_{\rn{}}(S_{8})$ to $\lie{Spin}{8} \subset \clifp{8}$. The negative spin representation $\Delta_{8}^{-} : \lie{Spin}{8} \rightarrow \lie{GL}{S_{8}^{-}}$ is defined similarly.
\end{definition}
In contrast to its eight-dimensional counterpart, the oriented volume element $\omega_{7} \in \clif{7}$ belongs to the center of $\clif{7}$ and satisfies $\omega_{7}^{2} = 1$. In particular, the elements $\pi^{\pm} = 1/2 (1 \pm \omega_{7}) \in \clif{7}$ define projection operators acting on $\clif{7}$ by left-multiplication. The associated subalgebras $\clifq{7}{\pm} = \pi^{\pm} \,\clif{7}$ induce a decomposition 
\begin{equation*}
\label{Eq.clif7quiral}
\clif{7} = \clifq{7}{+} \oplus \clifq{7}{-}.
\end{equation*}
Note that $\alpha(\omega_{7}) = -\omega_{7}$  under the antipodal involution $\alpha : \clif{7} \rightarrow \clif{7}$. It follows that the even part $\clifp{7}$ meets both $\clifq{7}{\pm}$ trivially. Hence, by restricting each projection $\pi^{\pm}$ to $\clifp{7}$ we get two injective linear transformations $\pi^{\pm}_{0} : \clifp{7} \rightarrow \clifq{7}{\pm}$ that induce an isomorphism $\pi_{0}^{-} \circ (\pi_{0}^{+})^{-1}$ from $\clifq{7}{+}$ to $\clifq{7}{-}$. Combining the isomorphism $P_{6} : \clif{6} \simeq \clifp{7}$ given by equation \cref{Eq.eveniso} with the identification $\clif{6} \simeq \rn{}(8)$~\cite[Theorem 4.3 p 27]{law.SG} we conclude that $\clifq{7}{+}$, $\clifq{7}{-}$ and $\clifp{7}$ are all isomorphic to the matrix algebra $\rn{}(8)$. In particular, $\clif{7} \simeq \rn{}(8) \oplus \rn{}(8)$ has two irreducible inequivalent real representations $c^{\pm} : \clif{7} \rightarrow \mathrm{End}_{\rn{}}(S_{8}^{\pm})$ that become equivalent when restricted to the even part $\clifp{7} \simeq \rn{}(8)$.
\begin{definition}
\label{Def.spin7}
The real spin representation $\Delta_{7} : \lie{Spin}{7} \rightarrow \lie{GL}{S_{8}^{+}}$ is obtained by restricting the representation $c^{+} :  \clif{7} \rightarrow \mathrm{End}_{\rn{}}(S_{8}^{+})$ to $\lie{Spin}{7} \subset \clifp{7}$.
\end{definition}
\begin{remark}
\label{Rmk.inner}
Let $S = S_{7}, S_{8}^{+}, S_{8}^{-}$ and denote by $\Delta$ the associated spin representation. Then $S$ admits an inner product $\braket{\cdot}{\cdot}$ such that Clifford multiplication by unit vectors in the corresponding Euclidean space is an orthogonal transformation. In particular, $\Delta$ acts by isometries of $S$. Given that $\lie{Spin}{7}$ and $\lie{Spin}{8}$ are connected, the image of $\Delta$ sits inside the connected component of $\lie{GL}{8}$. After fixing an orientation for $S$ we assume without loss of generality that the spin representations $\Delta_{7}$, $\Delta_{8}^{+}$ and $\Delta_{8}^{-}$ all take values in the special orthogonal group $\lie{SO}{8}$. The spinor space is the Clifford module $S$ endowed with such an orientation and inner product. Spinors in $S_{8}^{+}$ are called positive, or right-handed, while elements in $S_{8}^{-}$ are negative or left-handed spinors.
\end{remark}
\section{Isotropy subgroups of unit spinors}
\label{Sec.sphfibstab}
In this section we exhibit a copy of $\lie{Spin}{7}$ inside $\lie{Spin}{8}$ that identifies with the isotropy subgroup of a unit spinor of right chirality. A similar result can be obtained for left-handed spinors. For this, recall that the oriented volume element
\begin{equation*}
\label{Eq.vol8}
\omega_{8} = e_{1} \cdots e_{8} \in \lie{Spin}{8},
\end{equation*}
anti-commutes with vectors in $\rn{8}$ and satisfies $\omega_{8}^{2} = 1$. It follows that the fibre of the covering map $\mathrm{Ad}_{8} : \lie{Spin}{8} \rightarrow \lie{SO}{8}$ over $-\mathrm{id} \in \lie{SO}{8}$ is given by $\lbrace \omega_{8}, -\omega_{8} \rbrace$. 
\begin{lemma}
\label{Lem.isospin}
Let $\iota^{+} : \lie{Spin}{7} \rightarrow \lie{Spin}{8}$ be the lifting of the spin representation $\Delta_{7} : \lie{Spin}{7} \rightarrow \lie{SO}{8}$ along $\mathrm{Ad}_{8} : \lie{Spin}{8} \rightarrow \lie{SO}{8}$ specified by the condition 
\begin{equation}
\label{Eq.isospin}
\iota^{+}(-1) = \omega_{8}.
\end{equation}
Set $\sigma^{+} = \Delta_{8}^{+} \circ \iota^{+} : \lie{Spin}{8} \rightarrow \lie{SO}{8}$, where $\Delta_{8}^{+}$ is the right-handed spin representation introduced in \cref{Def.rightspin}. Then $\sigma^{+}$ is an embedding that factors though $\mathrm{Ad}_{7}$ via an injective homomorphism $\hat{\sigma}^{+} : \lie{SO}{7} \rightarrow \lie{SO}{8}$ with the following properties.
\begin{enumerate}[label={\alph*)}]
\item There is a right-handed spinor $\psi \in S_{8}^{+}$ with norm $\abs{\psi} = 1$ such that the image $\hat{\sigma}^{+}(\lie{SO}{7})$ coincides with the isotropy subgroup $\mathrm{Stab}(\psi)$.
\item The images of $\sigma^{+}$ and $\hat{\sigma}^{+}$ are exactly equal. In particular, $\sigma^{+}(\lie{Spin}{8}) = \mathrm{Stab}(\psi)$.
\end{enumerate}
\end{lemma}
\begin{proof}
The spin representations $\Delta_{7}$, $\Delta_{8}^{+}$ are faithful. Therefore, $\iota^{+}$ and $\sigma^{+}$ are both embeddings. To prove that $\sigma^{+}$ factors though $\mathrm{Ad}_{7}$, it is enough to check that $\sigma^{+}$ maps the generator $-1 \in \kr \mathrm{Ad}_{7}$ to the identity of $S_{8}^{+}$. This follows from the definition of $\sigma^{+}$, equation \cref{Eq.isospin}, and the fact that $\omega_{8} \in \lie{Spin}{8}$ acts as the identity on the positive spinor module $S_{8}^{+}$ (see \cref{Def.rightspin}). Set $\hat{\sigma}^{+}(A) = \sigma^{+}(\xi)$, where $\xi \in \lie{Spin}{7}$ is such that $\mathrm{Ad}_{7}(\xi) = A$. Note that $\hat{\sigma}^{+}$ is well defined by condition \cref{Eq.isospin}. To show that $\hat{\sigma}^{+}$ is injective, take $A \in \kr \hat{\sigma}^{+}$ and $\xi \in \lie{Spin}{7}$ such that $\mathrm{Ad}_{7}(\xi) = A$. Then the expression $\mathrm{id} = \hat{\sigma}^{+}(\mathrm{Ad}_{7}(\xi)) = \sigma^{+}(\xi)$ implies that $\xi$ is the identity element since $\sigma^{+}$ is one to one.

The group $\lie{SO}{7}$ has two irreducible real representations of dimension less than or equal to eight, namely, the trivial $1$-dimensional representation, and the fundamental representation of dimension seven. Hence, any decomposition of $\hat{\sigma}^{+}$ into irreducible representations is either a sum of $8$ copies of the trivial representation, or a sum of a single copy of the trivial representation and a copy of the fundamental representation. The first case is impossible since $\hat{\sigma}^{+}$ is injective. Therefore, there exists a non-zero positive spinor $\psi \in S_{8}^{+}$ corresponding to the trivial summand in the aforementioned decomposition into irreducibles that is fixed by the image of $\hat{\sigma}^{+}$. Without loss of generality we choose $\psi$ so that $\abs{\psi} = 1$. Finally, the images of $\sigma^{+}$ and $\hat{\sigma}^{+}$ coincide because $\sigma^{+} = \hat{\sigma}^{+} \circ \mathrm{Ad}_{7}$, and the vector representation $\mathrm{Ad}_{7}$ is surjective.
\end{proof}
\begin{proposition}
\label{Pro.isospin}
There is a unit spinor $\psi \in S_{8}^{+}$ such that the image of the embedding $\iota^{+} : \lie{Spin}{7} \rightarrow \lie{Spin}{8}$ defined in \cref{Lem.isospin} coincides with the isotropy subgroup $\mathrm{Stab}(\psi)$.
\end{proposition}
\begin{proof}
By \cref{Lem.isospin} there is a unit spinor $\psi \in S_{8}^{+}$ such that $\mathrm{Im}(\sigma^{+}) = \mathrm{Stab}(\psi)$. It follows that every element in the image of $\iota^{+}$ stabilizes $\psi$. Conversely, given $\gamma \in \lie{Spin}{8}$ that fixes a unit spinor $\psi \in S_{8}^{+}$ as in \cref{Lem.isospin}, there is an element $\zeta \in \lie{Spin}{7}$ such that $\Delta_{8}^{+}(\gamma) = \sigma^{+}(\zeta)$. On the other hand, $\sigma^{+} = \Delta_{8}^{+} \circ \iota^{+}$. It follows that $\gamma = \iota^{+}(\zeta)$ since $\Delta_{8}^{+}$ is faithful.
\end{proof}
\begin{remark}
\label{Rmk.bigdim}
The periodicity isomorphism $\clif{n + 8} \simeq \clif{n} \otimes \clif{8}$~\cite[Theorem 4.3]{law.SG} implies that the dimension of the positive spinor module $S_{7 + 8k}^{+}$ depends exponentially on $k$ as
\begin{equation*}
\label{Eq.bigdim}
\dim_{\rn{}} S_{7 + 8k}^{+} = 8 \cdot2^{4k}.
\end{equation*}
Thus, the analogue of \cref{Lem.isospin} in dimension $7 + 8k$ does not hold if $k \geq 1$. However, our argument can be easily adapted to prove the existence of a proper and non-trivial subspace of $S_{7 + 8k}^{+}$ fixed by the image of the corresponding embedding $\iota^{+} : \lie{Spin}{7 + 8k} \rightarrow \lie{SO}{8 \cdot2^{4k}}$.
\end{remark}
\begin{remark}
\label{Rmk.spinstrucsph}
There is another copy of $\lie{Spin}{7}$ inside $\lie{Spin}{8}$ characterized as the stabilizer of an element in the vector representation of $\lie{Spin}{8}$. To see this, let $e_{0} \in \rn{8}$ be the first coordinate vector. Define $\mathrm{ev}_{0} : \lie{SO}{8} \rightarrow S^{7}$ by $\mathrm{ev}_{0}(A) = Ae_{0}$. Then the isotropy subgroup $\lie{Stab}{e_{0}} \subset \lie{SO}{8}$ coincides with the fibre $\mathrm{ev}_{0}^{-1}(e_{0})$. Therefore, we have an isomorphism $\lie{SO}{7} \simeq \lie{Stab}{e_{0}}$ that yields a commutative diagram
\begin{equation}
\label{Eq.snspinstruct}
\begin{split}
\xymatrix@R=2mm@C=8mm{
 \lie{Spin}{7} \ar[r] \ar[dd]_{\mathrm{Ad}_{7}} & \lie{Spin}{8} \ar[dd]^{_{\mathrm{Ad}_{8}}} \ar[dr]^{\widetilde{\mathrm{ev}}_{0}} &  \\
 &  & S^{7}  \\
 \lie{SO}{7} \ar[r] &  \lie{SO}{8} \ar[ru]_{\mathrm{ev}_{0}} & 
}
\end{split}
\end{equation}
The bottom row is the principal $\lie{SO}{7}$-bundle of orthonormal frames of $S^{7}$. Thus, the top row is a principal $\lie{Spin}{7}$-bundle that defines a spin structure for $S^{7}$.
\end{remark}
\begin{remark}
\label{Rmk.G2}
The group $\lie{Spin}{7}$ acts transitively on the $7$-sphere formed by unit spinors of $S_{7}$ with isotropy given by the exceptional group $G_{2}$. In particular, there is a diffeomorphism
\begin{equation}
\label{Eq.G2Spin7}
\lie{Spin}{7} / G_{2} \simeq S^{7}.
\end{equation} 
Moreover, any $G_{2}$-subgroup of $\lie{Spin}{7}$ arises in this way~\cite[Theorem 3 Section 2]{vara.SP78.01}. The group $G_{2}$ can also be described as a subgroup of $\lie{Spin}{8}$. Denote by $\iota^{v} : \lie{Spin}{7} \rightarrow \lie{Spin}{8}$ the inclusion map in the top row of \cref{Eq.snspinstruct}. Then the images of $\lie{Spin}{7}$ under the embeddings $\iota^{v}, \iota^{+}$ are not conjugate subgroups of $\lie{Spin}{8}$. Therefore,~\cite[Theorem 5 Section 3]{vara.SP78.01} gives an isomorphism
\begin{equation}
\label{Eq.G2Spin8}
\iota^{v}\left( \lie{Spin}{7} \right) \cap \iota^{+}\left( \lie{Spin}{7} \right) \simeq G_{2}.
\end{equation}
In other words, there is copy of $G_{2}$ inside $\lie{Spin}{8}$ that fixes both a unit spinor of positive chirality and a unit element in the vector representation of $\lie{Spin}{7}$.
\end{remark}
\begin{proposition}
\label{Pro.difeoS7}
Consider the embedding $\iota^{+} : \lie{Spin}{7} \rightarrow \lie{Spin}{8}$ defined in \cref{Lem.isospin}. Then there is a diffeomorphism between the homogeneous space $\lie{Spin}{8} / \iota^{+} (\lie{Spin}{7})$ and the standard $7$-sphere $S^{7}$.
\end{proposition}
\begin{proof}
Consider the commutative diagram
\begin{equation}
\label{Eq.g2spin7}
\begin{split}
\xymatrix@R=8mm@C=8mm{
G_{2} \ar[r] \ar[d] & \iota^{v} (\lie{Spin}{7}) \ar[d] \ar[r] &  \iota^{v} (\lie{Spin}{7}) / G_{2} \ar[d]^{f}  \\
\iota^{+} (\lie{Spin}{7}) \ar[r] & \lie{Spin}{8} \ar[r] & \lie{Spin}{8} / \iota^{+} (\lie{Spin}{7})
}
\end{split}
\end{equation}
where all the arrows on the leftmost square of \cref{Eq.g2spin7} are inclusion maps. The right square in \cref{Eq.g2spin7} defines a smooth map $f$ between homogeneous manifolds. Equation \cref{Eq.G2Spin8} implies that $f$ is an injective immersion. Moreover, for dimensional reasons, the differential of $f$ is an isomorphism at every point of its domain. On the other hand, \cref{Eq.G2Spin7} induces a diffeomorphism between $\iota^{v} (\lie{Spin}{7}) / G_{2}$ and the standard $7$-sphere $S^{7}$. By compactness of $S^{7}$, $f$ is a proper injective immersion, and in particular a smooth embedding onto its image. The regular value theorem implies that $f$ is surjective and therefore a diffeomorphism.
\end{proof}
\begin{corollary}
\label{Cor.difeoRP7}
Thee is a diffeomorphism between the real projective space $\proj{R}{7}$ and the homogeneous space $\lie{SO}{8} / \Delta_{7}(\lie{Spin}{7})$, where $\Delta_{7} : \lie{Spin}{7} \rightarrow \lie{SO}{8}$ is the real spin representation as in \cref{Def.spin7}.
\end{corollary}
\begin{proof}
By definition, $\Delta_{7} = \mathrm{Ad}_{8} \circ \iota^{+}$. Hence, $\mathrm{Ad}_{8}$ descends to a well defined map $\tilde{f} : \lie{Spin}{8} / \iota^{+}(\lie{Spin}{7}) \rightarrow \lie{SO}{8} / \Delta_{7}(\lie{Spin}{7})$ that is compatible with the diffeomorphism $S^{7} \simeq \lie{Spin}{8} / \iota^{+}(\lie{Spin}{7})$ given by~\cref{Pro.difeoS7}. Since $\mathrm{Ad}_{8}(\xi) = \mathrm{Ad}_{8}(-\xi)$, $\tilde{f}$ is also compatible with the standard antipodal map $S^{7} \rightarrow \proj{R}{7}$.
\end{proof}
\section{The classification of $\lie{Spin}{7}$-structures in dimension eight}
\label{Sec.spin7}
\subsection{Existence and uniqueness of $G_{2}$-structures}
\label{Sub.g2struct}
The exceptional Lie group $G_{2} \subset \lie{SO}{7}$ is simply-connected. Consequently, it admits a lifting along the universal covering $\mathrm{Ad}_{7} : \lie{Spin}{7} \rightarrow \lie{SO}{7}$. Moreover, the image of any such lifting is given by the stabilizer of a unit spinor in $S_{7}$ (see \cref{Rmk.G2}). Arguing as in \cref{Cor.difeoRP7}, one shows that the vector representation $\mathrm{Ad}_{7}$ descends to a well defined map $\lie{Spin}{7} / G_{2} \rightarrow \lie{SO}{7}/ G_{2}$ that is equivalent to the standard double cover $S^{7} \rightarrow \proj{R}{7}$. This gives a homotopy equivalence between the real projective space $\proj{R}{7}$ and the homotopy fibre of the map $\bg{G_{2}} \rightarrow \bg{\lie{SO}{7}}$ induced by the canonical inclusion. From this we get the commutative diagram below, where the horizontal rows are Hurewicz fibrations.
\begin{equation}
\label{Eq.hurG2}
\begin{split}
\xymatrix@R=8mm@C=8mm{
S^{7} \ar[d]_{2:1} \ar[r] & \bg{G_{2}} \ar[r] \ar@{=}[d] & \bg{\lie{Spin}{7}} \ar[d]^{\mathrm{\bg{Ad}_{7}}} \\
\proj{R}{7} \ar[r] & \bg{G_{2}} \ar[r] & \bg{\lie{SO}{7}}  
}
\end{split}
\end{equation}
A $G_{2}$-structure on a compact and oriented $7$ manifold $M$ is a reduction of the structure group of the tangent bundle $TM$, from $\lie{SO}{7}$ down to the subgroup $G_{2}$. The commutative diagram \cref{Eq.hurG2} and \cref{Def.obscocycle} imply that the second Stiefel-Whitney class $w_{2}(M) \in \cohomg{2}{M}{\pi_{1}\proj{R}{7}}$ is the complete obstruction to the existence of a $G_{2}$-structure on $M$. In other words, an oriented $7$-manifold admits a $G_{2}$-structure if and only if it is spin.

Suppose that $M$ is spin, and let $S_{M}$ be the spinor bundle associated with some spin structure on $M$. Since $S_{M}$ has rank eight, a generic smooth section of $S_{M}$ vanishes in a submanifold of dimension $-1$. Therefore, the spinor bundle $S_{M}$ carries a nowhere vanishing spinor field that defines a $G_{2}$-structure on $M$, and any $G_{2}$-structure on $M$ arises in this way.
\begin{remark}
\label{Rmk.uniqG2}
The difference between two $G_{2}$-structures on $M$ is completely determined by the primary difference \cref{Eq.primdiff} that takes values in $\cohomg{7}{M}{\pi_{7}S^{7}} = \cohomg{7}{M}{\zn{}}$. The orientation on $M$ establishes an isomorphism between $\cohomg{7}{M}{\zn{}}$ and the group of integers $\zn{}$. Hence, \cref{Lem.transact} implies that $\zn{}$ acts freely and transitively on the set of $G_{2}$-structures on a spin $7$-manifold $M$~\cite[Lemma 1.1]{crownord.NIG2S.15}. 
\end{remark}
\subsection{Existence and uniqueness of $\lie{Spin}{7}$-structures}
\label{Sub.spin7struct}
Up to conjugacy, there are two $\lie{Spin}{7}$-subgroups inside $\lie{SO}{8}$~\cite[Theorem 3 Section 1]{vara.SP78.01}. A distinguished representative for one of them is given by the image of the real spin representation $\Delta_{7} : \lie{Spin}{7} \rightarrow \lie{SO}{8}$ introduced in \cref{Def.spin7}. Consequently, we can define a $\lie{Spin}{7}$-structure on an oriented $8$-manifold $W$ as a reduction of the structure group of the tangent bundle $TW$ to the subgroup $\Delta_{7}(\lie{Spin}{7}) \subset \lie{SO}{8}$. On the other hand, the embedding $\iota^{+} : \lie{Spin}{7} \rightarrow \lie{Spin}{8}$ defined in \cref{Lem.isospin} gives a particular lifting for $\Delta_{7}$. \cref{Pro.isospin} characterizes the image of $\iota^{+}$ as the stabilizer of a unit spinor in $S_{8}^{+}$. Therefore, a $\lie{Spin}{7}$-strucutre on $W$ yields a spin structure on $W$ with a distinguished reduction to the isotropy subgroup of a unit spinor of positive chirality.

Conversely, suppose that $W$ is spin and let $S_{W}^{+}$ be the positive spinor bundle of $W$. Recall that $S_{W}^{+}$ is obtained as an associated vector bundle
\begin{equation*}
\label{Eq.assocpos}
S_{W}^{+} = \slashed{P} \times_{\Delta_{8}^{+}} S_{8}^{+},
\end{equation*}
where $\slashed{P}$ is a principal $\lie{Spin}{8}$-bundle corresponding to a spin structure on $W$, and $\Delta_{8}^{+} : \lie{Spin}{8} \rightarrow \lie{SO}{8}$ is the positive spin representation introduced in \cref{Def.rightspin}. Then $S_{W}^{+}$ admits a nowhere vanishing section if and only if the Euler class $e(S_{W}^{+}) \in \cohomg{8}{W}{\zn{}}$ vanishes. This last condition is equivalent to the existence of a reduction of the structure group of $S_{W}^{+}$ to the subgroup $\iota^{+}(\lie{Spin}{7}) \subset \lie{Spin}{8}$.

In summary, $W$ admits a $\lie{Spin}{7}$-structure if and only if $W$ is spin and the Euler class of the associated positive spinor bundle vanishes~\cite[Theorem 3.4]{gragree.STTA.70}. 
\begin{proof}[Proof of \cref{Thm.mainA}]
Consider the embedding $\iota^{+} : \lie{Spin}{7} \rightarrow \lie{Spin}{8}$. By \cref{Pro.difeoS7} there is a diffeomorphism between the homogeneous space $\lie{Spin}{8} / \iota^{+} (\lie{Spin}{7})$ and the standard $7$-sphere. It follows that the homotopy fibre of the induced map $\bg{\iota}^{+}$ is homotopy equivalent to $S^{7}$. In particular, we have a Hurewicz fibration
\begin{equation*}
\label{Eq.spin7fib}
S^{7} \longrightarrow \bg{\lie{Spin}{7}} \stackrel{\bg{\iota}^{+}}{\longrightarrow} \bg{\lie{Spin}{8}}.
\end{equation*}
Let $W$ be a compact $8$-manifold endowed with a $\lie{Spin}{7}$-structure. Write $\partial W$ to denote the (possibly empty) boundary. Fix a spin structure and let $S^{+}: W \rightarrow \bg{\lie{Spin}{8}}$ be the classifying map corresponding to the positive spinor bundle $S^{+}_{W}$. According to \cref{Def.red}, a $\lie{Spin}{7}$-structure for $W$ is the homotopy class represented by a map $\hat{S} : W \rightarrow \lie{BSpin}{7}$ lifting $S^{+}$ along the projection $\bg{\iota}^{+}$ up to homotopy. The discussion in \cref{Sub.g2struct} implies that $\hat{S}$ restricts to a $G_{2}$-structure on the boundary $\partial W$.

By \cref{Def.primdiff}, the primary difference between two $\lie{Spin}{7}$-structures on $W$ that coincide on $\partial W$ is an element of $\cohomrelg{7}{W}{\partial W}{\zn{}}$. Poincar\'{e} duality and the Hurewicz homomorphism imply that the cohomology group $\cohomrelg{7}{W}{\partial W}{\zn{}}$ is isomorphic to the abelianization of the fundamental group $\pi_{1}W$. Thus, our hypothesis guarantee that $\cohomrelg{7}{W}{\partial W}{\zn{}} = 0$. Now it is clear that \cref{Thm.mainA} is a special case of \cref{Cor.ftransact} applied to the CW pair $(X, Y)= (W, \partial W)$.
\end{proof}
\begin{remark}
The spinor bundle of a compact spin $8$-manifold $W$ splits as the sum of positive and negative spinor bundles $S_{W}^{+}$, $S_{W}^{-}$. Clifford multiplication by a vector field $v \in \Gamma(TW)$ establishes an isomorphism $c_{v} : S_{W}^{+} \rightarrow S_{W}^{-}$ that reverts chirality of spinors. On the other hand, a $\lie{Spin}{7}$-structure on $W$ is characterized by a unit positive spinor field $\psi \in \Gamma(S_{W}^{+})$. In particular, $c_{v}( \psi) \in \Gamma(S_{W}^{-})$ vanishes over the zero-locus of $v$. This means that Clifford multiplication $c_{v}$ induces a natural pairing between $\lie{Spin}{7}$-structures on $W$ when the Euler characteristic of $W$ equals zero (eg if $\partial W = \emptyset$). By~\cite[Theorem 3.4]{gragree.STTA.70}, the Euler classes $e(S_{W}^{+})$, $e(S_{W}^{-})$ differ by $e(TW)$. Hence, the vanishing of any two $e(S_{W}^{+})$, $e(S_{W}^{+})$, $e(TW)$ implies the vanishing of the third. Therefore, \cref{Thm.mainB} is non-trivial only when $W$ is closed.
\end{remark}
A well known result by Milnor~\cite{mil.SSM.63} says that the spin cobordism group in dimension seven is trivial (see Stong~\cite[Chapter XI]{stong.NCT} for details). This means that any compact spin $7$-manifold $M$ is the spin boundary of a compact spin $8$-manifold $W$. Crowley and Nordstr\"{o}m improved this result by showing that there is always some $W$ with a $\lie{Spin}{7}$-structure~\cite[Lemma 3.4]{crownord.NIG2S.15}. Using this fact, the authors introduced a pair of invariants that completely determine the deformation class of a $G_{2}$-structure on $M$; that is, the class of a geometric $G_{2}$-structure identified under homotopy equivalence and diffeomorphism. This program ultimately shows that $M$ posses at least $24$ deformation classes of $G_{2}$-structures~\cite[Theorem 1.3]{crownord.NIG2S.15}. We expect to complete the classification of $\lie{Spin}{7}$-structures up to diffeomorphism in a future publication.
\bibliography{Bib_Spin7_G2}
\bibliographystyle{abbrv}
\end{document}